\documentclass[12pt, reqno]{amsart}
 \usepackage{amsmath, amsthm, amscd, amsfonts, amssymb, graphicx, color, float}
\usepackage[bookmarksnumbered, colorlinks, plainpages]{hyperref}
\input{mathrsfs.sty}
\hypersetup{colorlinks=true,linkcolor=red, anchorcolor=green, citecolor=cyan, urlcolor=red, filecolor=magenta, pdftoolbar=true}

\newtheorem{theorem}{Theorem}[section]
\newtheorem{lemma}[theorem]{Lemma}

\newtheorem{corollary}[theorem]{Corollary}
\theoremstyle{definition}

\theoremstyle{remark}

\numberwithin{equation}{section}

\begin{document}

\title{Noncommutative Blackwell--Ross martingale inequality}
\author[A. Talebi, M.S. Moslehian, Gh. Sadeghi ]{Ali Talebi$^1$, Mohammad Sal Moslehian $^1$ and Ghadir Sadeghi$^2$ }

\address{$^1$ Department of Pure Mathematics, Center of Excellence in Analysis on Algebraic Structures (CEAAS), Ferdowsi University of
Mashhad, P.O. Box 1159, Mashhad 91775, Iran}
\email{alitalebimath@yahoo.com}
\email{moslehian@um.ac.ir, moslehian@member.ams.org}

\address{$^2$ Department of Mathematics and Computer Sciences, Hakim Sabzevari University, P.O. Box 397, Sabzevar, Iran}
\email{ghadir54@gmail.com, g.sadeghi@hsu.ac.ir}

%======================================================================

\begin{abstract}
We establish a noncommutative Blackwell--Ross inequality for supermartingales under a suitable condition which generalize Khan's works to the noncommutative setting. We then employ it to deduce an Azuma-type inequality.
\end{abstract}

\maketitle

%==============================================================================================================================================================
\section{Introduction}

Blackwell \cite{BLA} showed that if $\{X_n, X_0=0, n \geq 0\}$ is a martingale such that $|X_n-X_{n-1}|\leq \alpha$ for all $n$, then for each positive number $c$,
$${\rm Prob}(X_n\geq mc {\rm ~for~some~} n\geq m) \leq \exp\left(\frac{-mc^2}{2\alpha^2}\right),$$
which gives a generalization of a result of Hoeffding \cite{HOE}.  Ross \cite{ROS} extended Blakwell's result to the case where the bound on the martingale difference is not symmetric. Indeed, Ross employed a supermartingale argument to show that the same is true when $-\alpha\leq X_n-X_{n-1} \leq \beta$, where $\alpha, \beta >0$. Khan \cite{KHA} generalized Blackwell--Ross inequality for martingales (supermartingales) under a subnormal structure on the conditional moment generating function $\varphi_n(\theta)=\mathbb{E}(\exp(\theta X_n)|\mathcal{F}_{n-1})$ subject to some mild conditions.

In this paper, we adopt the classical ideas in probability theory and the Golden--Thompson inequality to establish a Blackwell--Ross martingale inequality under a non-symmetric bound on the martingales differences in the framework of noncommutative probability spaces.

A von Neumann algebra $\mathfrak{M}$ on a Hilbert space $\mathfrak{H}$ with unit element $1$ equipped with a normal faithful tracial state $\tau:\mathfrak{M}\to \mathbb{C}$ is called a noncommutative probability space. We denote by $\leq$ the usual order on the self-adjoint part $\mathfrak{M}^{sa}$ of $\mathfrak{M}$.
For each self-adjoint operator $x\in \mathfrak{M}$, there exists a unique spectral measure $E$ as a $\sigma$-additive mapping with respect to the strong operator topology from the Borel $\sigma$-algebra $\mathcal{B}(\mathbb{R})$ of $\mathbb{R}$ into the set of all orthogonal projections such that for every Borel function $f: \sigma(x)\to \mathbb{C}$ the operator $f(x)$ is defined by $f(x)=\int f(\lambda)dE(\lambda)$, in particular, $\textbf{1}_B(x)=\int_BdE(\lambda)=E(B)$.

The celebrated Golden--Thompson inequality \cite{RUS} states that for any self-adjoint elements $y_1, y_2$ in a noncommutative probability space $ \mathfrak{M}$, the inequality
\begin{eqnarray}\label{T2}
\tau(e^{y_1+y_2})\leq \tau(e^{y_1} e^{y_2})
\end{eqnarray}
holds; see also \cite{SUK} for some Golden--Thompson type inequalities.

For $p\geq1$, the noncommutative $L^P$-space $L^p(\mathfrak{M})$ is defined as the completion of $\mathfrak{M}$ with respect to the $L^p$-norm $\|x\|_p:=\left(\tau(|x|^p)\right)^{1/p}$. The commutative cases of discussed spaces are usual $L^p$-spaces. For further information we refer the reader to \cite{P} and references therein.\\
Let $\mathfrak{N}$ be a von Neumann subalgebra of $\mathfrak{M}$. Then there exists a normal positive contractive projection   $\mathcal{E}_{\mathfrak{N}}:\mathfrak{M}\to\mathfrak{N}$ satisfying the following properties:\\
(i) $\mathcal{E}_{\mathfrak{N}}(axb) =a\mathcal{E}_{\mathfrak{N}}(x)b$ for any $x\in\mathfrak{M}$ and $a, b\in\mathfrak{N}$;\\
(ii) $\tau\circ\mathcal{E}_{\mathfrak{N}}=\tau$.\\
Moreover, $\mathcal{E}_{\mathfrak{N}}$ is the unique mapping satisfying (i) and (ii). The mapping $\mathcal{E}_{\mathfrak{N}}$ is called the conditional expectation of $\mathfrak{M}$ with respect to $\mathfrak{N}$.

Let $\mathfrak{N}\subseteq \mathfrak{A}_j\,\,(1\leq j\leq n)$ be von Neumann subalgebras of $\mathfrak{M}$.
We say that the $\mathfrak{A}_j$ are order independent over $\mathfrak{N}$ if for every $2\leq j\leq n$, the equality $$\mathcal{E}_{j-1}(x)=\mathcal{E}_{\mathfrak{N}}(x)$$ holds for all $x\in\mathfrak{A}_j$, where $\mathcal{E}_{j-1}$ is the conditional expectation of $\mathfrak{M}$ with respect to the von Neumann subalgebra generated by $\mathfrak{A}_1,\ldots,\mathfrak{A}_{j-1}$; cf. \cite{JX}. Note that this notion of independence implies that $\mathfrak{N}$ should be the intersection of all $\mathfrak{A}_j$. In fact, if $x \in \mathfrak{A}_{j-1} \cap \mathfrak{A}_j$, then
\begin{align*}
x = \mathcal{E}_{j-1}(x) = \mathcal{E}_{\mathfrak{N}}(x) \in \mathfrak{N}.
\end{align*}

A filtration of $\mathfrak{M}$ is an increasing sequence $(\mathfrak{M}_j, \mathcal{E}_j)_{0\leq j\leq n}$ of von Neumann subalgebras of $\mathfrak{M}$ together with the conditional expectations $ \mathcal{E}_j$ of $\mathfrak{M}$ with respect to $\mathfrak{M}_j$ such that $\bigcup_j\mathfrak{M}_j$ is $w^*$--dense in $\mathfrak{M}$. It follows from $\mathfrak{M}_j\subseteq \mathfrak{M}_{j+1}$ that
\begin{eqnarray}\label{maral}
\mathcal{E}_i\circ\mathcal{E}_j=\mathcal{E}_j\circ\mathcal{E}_i=\mathcal{E}_{\min\{i,j\}}\,
\end{eqnarray}
for all $i,j\geq 0$. Generally, a sequence $(x_j)_{j\geq0}$ in $L^1(\mathfrak{M})$ is called a martingale (supermartingale, resp.) with respect to the filtration $(\mathfrak{M}_j)_{0\leq j\leq n}$ if $x_j \in L^1(\mathfrak{M}_j)$ and $\mathcal{E}_j(x_{j+1})=x_j$ ($\mathcal{E}_j(x_{j+1})\leq x_j$, resp.) for every $j\geq 0$. It follows from \eqref{maral} that $\mathcal{E}_j(x_i)=x_j$ for all $i\geq j$. Put $dx_j=x_j-x_{j-1}\,\,(j\geq 0)$ with the convention that $x_{-1}=0$. Then $dx=(dx_j)_{j\geq 0}$ is called the martingale difference of $(x_j)$. The reader is referred to \cite{Xu2} for more information.

%----------------------------------------------------------------------------------------------------------------------------------------------------------------%

\section{Main Results}

In this section, we provide a noncommutative Blackwell--Ross inequality. To this end, we will need the following lemma which was proved by Alon, et al. \cite{PE}.
\begin{lemma}\label{IPE}
For $0\leq\lambda\leq1$
\begin{eqnarray*}
\lambda e^{(1-\lambda)x}+(1-\lambda)e^{-\lambda x}\leq e^{\frac{x^2}{8}}\,.
\end{eqnarray*}
\end{lemma}
%----------------------------------------------------------------------------------------------------------------------------------------------------------------%
We are inspired by some ideas in the commutative case, e.g. \cite{KHA}, to provide our main result.

\begin{theorem}\label{main2}
Let $\left\{s_n=\sum_{i=1}^nx_n,~~ n\geq0\right\}$ be a self-adjoint supermartingale in $\mathfrak{M}$ with respect to a filtration $(\mathfrak{M}_n,\mathcal{E}_n)_{n\geq0}$ such that
\begin{eqnarray}\label{eq31}
\mathcal{E}_{n-1}(e^{t x_n})\leq f(t)\leq e^{-\gamma t+\lambda t^2}, \quad (\lambda>0, \gamma\geq0, t>0)
\end{eqnarray}
where $f(t)$ is a continuous positive function on $[0,\infty)$. Then for positive numbers $a$ and $b$, there exists $i\geq 1$ such that for any positive integer $m$
\begin{eqnarray}\label{eq32}
\tau\left( \vee_{n=m+i}^\infty \textbf{1}_{(a+bn, \infty)} (s_n) \right) \leq A^m e^{\frac{-a(b+\gamma)}{\lambda}},
\end{eqnarray}
where $A=e^{-bt_0}f(t_0) \leq 1$ and $t_0=\frac{b+\gamma}{\lambda}$. Moreover,
\begin{eqnarray}\label{eq33}
\tau\left( \vee_{n=m+j}^\infty \textbf{1}_{(bn, \infty)} (s_n) \right) \leq A_0^m e^{\frac{-m(b+\gamma)^2}{4\lambda}}
\end{eqnarray}
for some $j \geq 1$, where $A_0 =e^{-\frac{1}{2}(b-\gamma)t_0}f(t_0)$ and $t_0=\frac{b+ \gamma}{2\lambda}$.
\end{theorem}
\begin{proof}
Let $y_n= \exp(t s_n-at-bnt)$, $t>0$. We show that the sequence $(y_n)_{n\geq0}$ satisfies the following inequality at $t=t_0=\frac{b+\gamma}{\lambda}$.
To this end, note that
\begin{eqnarray*}
\tau(y_n)&=&\tau(\exp(t s_n-at-bnt))\\
&=&\tau(\exp(ts_{n-1}-at-b(n-1)t-bt+tx_n))\\
&\leq&\tau(\exp(ts_{n-1}-at-b(n-1)t)\exp(-bt+tx_n)) \qquad\quad (\text{by} ~\eqref{T2})\\
&=&\tau(\mathcal{E}_{n-1}(\exp(ts_{n-1}-at-b(n-1)t)\exp(-bt+tx_n)))\\
&=&\tau(\exp(ts_{n-1}-at-b(n-1)t)\mathcal{E}_{n-1}(\exp(-bt+tx_n)))\\
&=&\tau(y_{n-1}\mathcal{E}_{n-1}(\exp(-bt+tx_n)))\\
&=&\tau(y_{n-1}\mathcal{E}_{n-1}(e^{-bt}e^{tx_n}))\\
&=&e^{-bt}\tau(y_{n-1}\mathcal{E}_{n-1}(e^{t x_n}))\\
&=&e^{-bt}\tau(y_{n-1}^{\frac{1}{2}}\mathcal{E}_{n-1}(e^{t x_n})y_{n-1}^{\frac{1}{2}})\\
&\leq& e^{-bt}f(t)\tau(y_{n-1})\\
&=& A \tau(y_{n-1})\\
&\leq&\tau(y_{n-1})e^{-(b+\gamma)t+\lambda t^2}
\end{eqnarray*}
if $A = e^{-bt}f(t)$ and $t=t_0=\frac{b+\gamma}{\lambda}$, in which the first and second inequalities follows from (\ref{eq31}).\\
We have $\vee_{n=m+i}^{k}\textbf{1}_{[a+bn,\infty)}(s_n)\preceq\vee_{n=m+i}^{k}\textbf{1}_{[1,\infty)}(y_n)$, for every $k \geq m+i$, in which $m, i$ are positive integers, since
\begin{eqnarray*}
\vee_{n=m+i}^{k}\textbf{1}_{[a+bn,\infty)}(s_n)\wedge(\wedge_{n=m+i}^{k}\textbf{1}_{[0,1)}(y_n))=0.
\end{eqnarray*}
To show this assume that $\xi$  is an unit element in
$$\bigcup_{n=m+i}^k \textbf{1}_{[a+bn,\infty)}(s_n)(\mathfrak{H}) \bigcap \left(\bigcap_{n=m+i}^k \textbf{1}_{[0,1)}(y_n)(\mathfrak{H}) \right).$$
Therefore $\langle s_j \xi, \xi \rangle \geq a + bj$ and $\langle e^{c(s_j-a-bj)} \xi, \xi \rangle < 1$, for some $m+i \leq j \leq k$. By the operator version of the classical Jensen's inequality for the convex function $t \mapsto e^{c(t-a-bj)}$, we get
\begin{eqnarray*}
e^{\langle c(s_j-a-bj)(\xi), \xi \rangle} \leq \langle e^{c(s_j-a-bj)} \xi, \xi \rangle < 1.
\end{eqnarray*}
Consequently, $\langle c(s_j-a-bj)(\xi), \xi \rangle < 0$ and hence $\langle s_j\xi, \xi \rangle < a+ bj$ which gives rise to a contradiction.
Choose $i \in \mathbb{N}$ such that $\frac{A^{i}}{1 - A} \leq 1$.
Hence
\begin{eqnarray*}
\tau(\vee_{n=m+i}^\infty \textbf{1}_{[a+bn,\infty)}(s_n))&\leq &\tau(\vee_{n=m+i}^\infty \textbf{1}_{[1,\infty)}(y_n))\\
& \leq & \sum_{n=m+i}^\infty \tau(y_n)\\
&\leq & \sum_{n=m+i}^\infty A \tau (y_{n-1})\\
& \vdots \\
&\leq & \sum_{n=m+i}^\infty A^{n}\tau(y_0)\\
&\leq & A^m e^{-at_0}
\end{eqnarray*}
for any positive integer $m$, and this ensures (\ref{eq32}).

To prove (\ref{eq33}), let $g(\alpha,n)=m(b-\alpha)+\alpha n$, $n\geq m$, $\alpha\leq b$ and note that $bn\geq g(\alpha,n)$ for every $n\geq m$. A minimization consideration leads to
the choice of $\alpha=\alpha_0=\frac{b-\gamma}{2}$. Thus
\begin{eqnarray*}
\textbf{1}_{[bn,\infty)}(s_n)\leq \textbf{1}_{[\frac{m(b+\gamma)}{2}+\frac{n(b-\gamma)}{2}, \infty)}(s_n)
\end{eqnarray*}
for any $n\geq m$.
From (\ref{eq32}) we infer that
\begin{eqnarray*}
\tau\left(\vee_{n=m+j}^\infty \textbf{1}_{[\frac{m(b+\gamma)}{2}+\frac{n(b-\gamma)}{2},\infty)}(s_n)\right) \leq A_0^m e^{\frac{-m(b+\gamma)^2}{4\lambda}},
\end{eqnarray*}
for some $j \geq 1$, where $A_0=e^{\frac{-1}{2}(b-\gamma)t_0}f(t_0)$ and $t_0=\frac{b+ \gamma}{2\lambda}$.
Hence
\begin{eqnarray*}
\tau\left(\vee_{n=m+j}^\infty \textbf{1}_{[bn,\infty)(s_n)}\right) \leq \tau\left(\vee_{n=m+j}^\infty \textbf{1}_{[\frac{m(b+r)}{2}+\frac{n(b-r)}{2}, \infty)}\right) \leq A_0^m e^{\frac{-m(b+\gamma)^2}{4\lambda}},
\end{eqnarray*}
which implies (\ref{eq33}).
\end{proof}

%-------------------------------------------------------------------------------------------------------------------------------------------------------------%
Note that, in view of the Jensen inequality for conditional expectations in the above Theorem, we lead to the following inequality:
\begin{align*}
\mathcal{E}_{n-1}(x_n) \leq -\gamma.
\end{align*}
This special case have investigated by Khan. Similar to arguments in \cite{KHA}, we may conclude that if $\left\{s_n=\sum_{i=1}^nx_n,~~ n\geq0\right\}$ is a self-adjoint supermartingale with respect to a filtration $(\mathfrak{M}_n,\mathcal{E}_n)_{n\geq0}$ such that $-\alpha \leq x_n \leq \beta$ and $\mathcal{E}_{n-1}(x_n)\leq -\gamma$ ~$(\alpha > \lambda \geq 0, \beta >0)$ for all $n$, then
for positive numbers $a$ and $b$, there exists $i\geq 1$ such that for any positive integer $m$
\begin{eqnarray*}
\tau\left( \vee_{n=m+i}^\infty \textbf{1}_{(a+bn, \infty)} (s_n) \right) \leq A^m e^{\frac{-8a(b+\gamma)}{(\alpha + \beta)^2}},
\end{eqnarray*}
in which $A=e^{-bt_0}f(t_0)$ and $t_0=\frac{8(b+\gamma)}{(\alpha + \beta)^2}$, where $f(t)= e^{-\gamma t}\left( pe^{(\alpha + \beta)tq} + qe^{-(\alpha + \beta)tp}\right)$ and $p = \frac{\alpha - \gamma}{\alpha + \beta}$ and $\frac{\beta + \gamma}{\alpha + \beta}$. Similarly,
\begin{eqnarray*}
\tau\left( \vee_{n=m+j}^\infty \textbf{1}_{(bn, \infty)} (s_n) \right) \leq A_0^m e^{\frac{-2m(b+\gamma)^2}{(\alpha + \beta)^2}}
\end{eqnarray*}
for some $j \geq 1$, where $A_0 =e^{-\frac{1}{2}(b-\gamma)t_0}f(t_0)$ and $t_0=\frac{4(b+ \gamma)}{(\alpha + \beta)^2}$.

%------------------------------------------------------------------------------------------------------------------------------------------------------------%

\begin{corollary} (Noncommutative Blackwell--Ross inequality)\label{main1}
Let $x=(x_j)_{0\leq j\leq n}$ be a self-adjoint martingale in $\mathfrak{M}$ with respect to a filtration $(\mathfrak{M}_j, \mathcal{E}_j)_{0\leq j\leq n}$ with $x_0=0$ and $dx_j=x_j-x_{j-1}$ be its associated martingale difference. Assume that $-\alpha\leq dx_j\leq \beta$ for some positive constants $\alpha, \beta\,\, (1\leq j\leq n)$. Then for any positive values $a,b$, there exists $i\geq 1$ such that for any positive integer $m$
\begin{eqnarray*}
\tau\left(\vee_{n=m+i}^{\infty}\textbf{1}_{[a+bn,\infty)}(x_n)\right)\leq A^m \exp\left\{\frac{-8 ab}{(\alpha+\beta)^2}\right\},
\end{eqnarray*}
where
\begin{eqnarray*}
A=\frac{\beta}{\alpha+\beta}\exp\left\{\frac{-8b(b+\alpha)}{(\alpha+\beta)^2}\right\}+
\frac{\alpha}{\alpha+\beta}\exp\left\{\frac{-8b(b-\beta)}{(\alpha+\beta)^2}\right\}\leq 1.
\end{eqnarray*}
Moreover,
\begin{eqnarray*}
\tau(\wedge_{m=1}^{\infty}\vee_{n=m}^{\infty}\textbf{1}_{[1,\infty)}(x_n))=\lim_{m\to\infty}\tau(\vee_{n=m}^{\infty}\textbf{1}_{[1,\infty)}(x_n))=0
\end{eqnarray*}
\end{corollary}
\begin{proof}
Note that $x_n = \sum_{k=1}^n dx_k$ for all $n$. Let $t>0$. The function $s\mapsto e^{ts}$ is convex, therefore for any $-\alpha\leq s\leq \beta$,
\begin{eqnarray*}
e^{st}\leq e^{t\beta}\frac{s+\alpha}{\alpha+\beta}+e^{-t\alpha}\frac{\beta-s}{\alpha+\beta}\,.
\end{eqnarray*}
Since $-\alpha\leq dx_j\leq \beta$, by the functional calculus, we have
\begin{eqnarray*}
e^{tdx_j}\leq e^{t\beta}\frac{dx_j+\alpha}{\alpha+\beta}+e^{-t\alpha}\frac{\beta-dx_j}{\alpha+\beta}\,.
\end{eqnarray*}
Since $\mathcal{E}_{j-1}$ is a positive map and $\mathcal{E}_{j-1}(dx_j)=0$, we reach
\begin{eqnarray*}
\mathcal{E}_{j-1}(e^{tdx_j})\leq e^{t\beta}\frac{\alpha}{\alpha+\beta}+e^{-t\alpha}\frac{\beta}{\alpha+\beta} \leq  e^{\frac{t^2(\alpha+\beta)^2}{8}},
\end{eqnarray*}
where the second inequality is deduced from Lemma \ref{IPE}  with $\lambda=\frac{\alpha}{\alpha+\beta}$ and $x=c(\alpha+\beta)$.
Hence the desired result can be deduced from Theorem \ref{main2} with $f(t) = e^{t\beta}\frac{\alpha}{\alpha+\beta}+e^{-t\alpha}\frac{\beta}{\alpha+\beta}$, $\gamma = 0$ and $\lambda = \frac{(\alpha+\beta)^2}{8}$.
\end{proof}
%----------------------------------------------------------------------------------------------------------------------------------------------------------------%
The authors of  \cite{SM1, SM2} proved a noncommutative Azuma-type inequality for noncommutative martingales in noncommutative probability spaces, and as applications, the authors obtained a noncommutative Heoffding inequality.
In the next corollary we give a noncommutative Blackwell inequality from which we deduce an extension of commutative Azuma-type inequality. One may regard the following conclusion as a stronger result than the noncommutative Azuma-type inequality.
\begin{corollary}\label{Cor1}
Let $x=(x_j)_{0\leq j\leq n}$ be a self-adjoint martingale in $\mathfrak{M}$ with respect to a filtration $(\mathfrak{M}_j, \mathcal{E}_j)_{0\leq j\leq n}$ and $dx_j=x_j-x_{j-1}$ be its associated martingale difference. Assume that $-\alpha\leq dx_j\leq \beta$ for some nonnegative constants $\alpha, \beta>0\,\, (1\leq j\leq n)$. Then for any positive value $c$, there exists $i\geq 1$ such that for any positive integer $m$
\begin{eqnarray*}
\tau\left( \vee_{n=m+i}^\infty \textbf{1}_{(cn, \infty)} (x_n) \right) \leq B^m \exp\left\{\frac{-2 mc^2}{(\alpha+\beta)^2}\right\},
\end{eqnarray*}
where
\begin{eqnarray*}
B=\frac{\beta}{\alpha+\beta}\exp\left\{\frac{-2c(c+2\alpha)}{(\alpha+\beta)^2}\right\}+
\frac{\alpha}{\alpha+\beta}\exp\left\{\frac{-2c(c-2\beta)}{(\alpha+\beta)^2}\right\}\leq 1.
\end{eqnarray*}
\end{corollary}
\begin{proof}
For $a=\frac{mc}{2}$ and $b=\frac{c}{2}$, it follows form Corollary \ref{main1} that
\begin{eqnarray}\label{IE2}
\tau\left(\vee_{n=m+i}^\infty \mathbf{1}_{[\frac{mc}{2}+\frac{nc}{2}, \infty)}(x_n)\right)\leq B^m \exp\left\{\frac{-8 mc^2}{4(\alpha+\beta)^2}\right\}
\end{eqnarray}
for some $i \geq 1$. Moreover, we have
\begin{eqnarray}\label{IE3}
\mathbf{1}_{[nc,\infty)}(x_n)\leq\mathbf{1}_{[\frac{mc}{2}+\frac{nc}{2}, \infty)}(x_n)
\end{eqnarray}
for every $n\in \mathbb{N}$.
Hence the result is deduced from (\ref{IE2}) and (\ref{IE3}).
\end{proof}
%----------------------------------------------------------------------------------------------------------------------------------------------------------------%
\begin{corollary}[Azuma-type inequality]
Let $Z_n$, ~ $n \geq 0$ be a martingale sequence of bounded random variables with respect to a filtration $(\mathcal{F}_n, \mathbb{E}_n)_{n\geq 1}$ on a probability space $(\Omega, \mathcal{F}, \mathbb{P})$ with $Z_0 = 0$. If $-\alpha \leq dZ_n \leq \alpha$ for all $n$, then for each $c > 0$ there exists $i\geq 1$ such that for any positive integer $m$
\begin{equation*}
\mathbb{P}(Z_n \geq nc {\rm ~for~some~} n\geq m+i)\leq \exp\{\frac{-mc^2}{2 \alpha^2}\}.
\end{equation*}
\end{corollary}
\begin{proof}
It immediately follows from Corollary \ref{Cor1}.
\end{proof}
%-------------------------------------------------------------------------------------------------------------------------------------------------------------%

Now we can state a version of classical Blackwell-Ross supermartingale inequality as follows; cf. \cite{KHA}.
\begin{corollary}
Let $\left\{S_n=\sum_{i=1}^n X_n,~~ n\geq0\right\}$ be a supermartingale of bounded random variables with respect to a filtration $(\mathcal{F}_n, \mathbb{E}_n)_{n=1}^N$ on a probability space $(\Omega, \mathcal{F}, \mathbb{P})$ such that
\begin{eqnarray*}
\mathbb{E}_{n-1}(e^{t X_n})\leq f(t)\leq e^{-\gamma t+\lambda t^2} \quad (\lambda>0, \gamma\geq0, t>0),
\end{eqnarray*}
where $f(t)$ is a positive continuous function. Then for positive numbers $a$ and $b$, there exists $i\geq 1$ such that for any positive integer $m$,
\begin{eqnarray*}
\mathbb{P}\left(S_n\geq a+bn {\rm ~for~some~} n\geq m+i\right) \leq A^m e^{\frac{-a(b+\gamma)}{\lambda}},
\end{eqnarray*}
where $A=e^{-bt_0}f(t_0) < 1$ and $t_0=\frac{b+\gamma}{\lambda}$. Moreover,
\begin{eqnarray*}
\mathbb{P}\left(S_n\geq bn {\rm ~for~some~} n\geq m+j\right) \leq A_0^m e^{\frac{-m(b+\gamma)^2}{4\lambda}}
\end{eqnarray*}
for some $j \geq 1$, where $A_0 =e^{-\frac{1}{2}(b-\gamma)t_0}f(t_0)$ and $t_0=\frac{b+ \gamma}{2\lambda}$.
\end{corollary}
%---------------------------------------------------------------------------------------------------------------------------------------------------------------%
\begin{corollary}
Let $\mathfrak{N}\subseteq\mathfrak{A}_j(\subseteq\mathfrak{M})$ be order independent over $\mathfrak{N}$. Let $x_j\in\mathfrak{A}_j$ be self-adjoint such that $\mathcal{E}_{\mathfrak{N}}(x_j)\leq0$ and
\begin{eqnarray*}
\mathcal{E}_{n-1}(e^{t x_n})\leq f(t)\leq e^{-\gamma t+\lambda t^2}, \quad (\gamma\geq0, \lambda>0, t>0),
\end{eqnarray*}
where $f(t)$ is a continuous positive function on $[0,\infty)$ such that $f(0)=1$. Then for positive numbers $a$ and $b$, there exists $i\geq 1$ such that for any positive integer $m$
\begin{eqnarray*}
\tau \left(\vee_{n=m+i}^\infty \textbf{1}_{[a+bn, \infty)} (s_n) \right)\leq A^m e^{\frac{-a(b+\gamma)}{\lambda}}.
\end{eqnarray*}

\end{corollary}
\begin{proof}
Let $\mathfrak{M}_0=\mathfrak{N}$ and $\mathcal{E}_0=\mathcal{E}_\mathfrak{N}$. For every $1\leq j\leq n$, let $\mathfrak{M}_j$ be the von Neumann subalgebra generated by $\mathfrak{A}_1,\ldots,\mathfrak{A}_{j-1}$ and $\mathcal{E}_j$ be the corresponding conditional expectation. Put $s_0:=0$ and $s_j:=\sum_{k=1}^jx_k$ for $1\leq j\leq n$. Then
$$\mathcal{E}_{j-1}(s_j)=\sum_{k=1}^{j-1}x_k+\mathcal{E}_{j-1}(x_k)=
\sum_{k=1}^{j-1}x_k+\mathcal{E}_{\mathfrak{N}}(x_k)\leq s_{j-1}.$$
It follows that $(s_j)_{0\leq j\leq n}$ is a supermartingale with respect to the filtration $(\mathfrak{M}_j, \mathcal{E}_j)_{0\leq j\leq n}$. Hence, the result follows via
$\left\{s_n=\sum_{i=1}^nx_n,~~ n\geq0\right\}$ in Theorem \ref{main2}.
\end{proof}
\bigskip

\textbf{Acknowledgement.} The authors are grateful to Prof. Fedor Sukochev for his useful comments.
%----------------------------------------------------------------------------------------------------------------------------------------------------------------%

\end{document}